\documentclass[letterpaper, 12pt]{amsart}

\usepackage[all]{xy}
\usepackage{latexsym}
\usepackage{amssymb}
\usepackage{fullpage}
\usepackage{amsmath}
\usepackage{amsthm}
\usepackage[dvips]{graphicx}
\usepackage{yfonts}
\usepackage{comment}
\newtheorem{theorem}{Theorem}[section]
\newtheorem{lemma}[theorem]{Lemma}

\newtheorem{defn}[theorem]{Definition}
\newtheorem{proposition}[theorem]{Proposition}

\newtheorem{question}[theorem]{Question}

\newtheorem{cor}[theorem]{Corollary}
\numberwithin{equation}{section}

\newcommand{\E}{\mathbb{E}}
\newcommand{\V}{\Vert}

\newcommand{\la}{\langle}
\newcommand{\ra}{\rangle}

\title{Free Products and the Lack of State preserving approximations of Nuclear C*-algebras}
\author{Caleb Eckhardt}
\address{Department of Mathematics, Purdue University, West Lafayette, IN, 47906}
\email{caeckhar@math.purdue.edu}
\date{}

\begin{document}
\begin{abstract} Let $A$ be a homogeneous C*-algebra and $\phi$ a state on $A.$  We show that if $\phi$ satisfies a certain faithfulness condition, then there is a net of finite-rank, unital completely positive, $\phi$-preserving maps on $A$ that tend to the identity pointwise. This combined with results of Ricard and Xu show that the reduced free product of homogeneous C*-algebras with respect to these states have the completely contractive approximation property.  We also give an example of a faithful state on $M_2\otimes C[0,1]$ for which no such state-preserving approximation of the identity map exists, thus answering a question of Ricard and Xu.  
\end{abstract}
\thanks{A portion of this work was completed while the author was funded by the research program ANR-06-BLAN-0015}
\maketitle
\section{Introduction}
It is not  completely understood which approximation properties of C*-algebras are preserved by reduced free products. One of the most satisfying positive results was obtained by  Dykema \cite{Dykema04} when he showed that exactness is preserved by reduced free products. On the other hand the non-nuclear C*-algebra $C^*_r(\mathbb{F}_2)$ is isomorphic to $C^*_r(\mathbb{Z})\ast_r C^*_r(\mathbb{Z})$ (with respect to the canonical trace), hence nuclearity is not preserved by reduced free products.  But J. de Canni{\`e}re and  Haagerup showed in \cite{Canniere85} that $C^*_r(\mathbb{F}_2)$ does have the completely contractive approximation property (CCAP). Recall that a C*-algebra has the CCAP if there is a net of finite rank, complete contractions that converge to the identity pointwise, and that the CCAP is strictly weaker than nuclearity, yet strictly stronger than exactness.  Bo{\.z}ejko and Picardello extended this result in \cite{Bozejko93} by showing that the reduced C*-algebras of free products of amenable groups have the CCAP, and Ricard and Xu \cite{Ricard06} extended this to the case of weakly amenable groups (with constant 1). Weak amenability (with constant 1) is the group theoretic analog of the CCAP, hence it is  natural to ask if the CCAP is preserved by reduced free products. This problem is still open, in fact it is unknown if the reduced free product of nuclear C*-algebras has the CCAP.
 
The goals of this paper are (i) to show that reduced free products of homogeneous C*-algebras with respect to certain well-behaved states have the CCAP and  (ii) to answer a question of Ricard and Xu by giving examples of states on nuclear C*-algebras that are not CP-approximable: 

\begin{defn} Let $A$ be a unital, nuclear C*-algebra and $\phi$ a state on $A.$ We say that $\phi$ is \textbf{CP-approximable} if there is a net $(T_\alpha)$ of finite-rank, unital, completely positive maps on $A$ such that $T_\alpha$ converges to the identity pointwise and $\phi\circ T_\alpha=\phi$ for all $\alpha.$
\end{defn}
Our goals are related by the following theorem of Ricard and Xu:
\begin{theorem}[Ricard \& Xu, \cite{Ricard06}] \label{thm:ericandxu} Let $(A_i,\phi_i)$ be a family of nuclear C*-algebras with states $\phi_i$ such that the corresponding GNS representations are faithful.  If each $\phi_i$ is CP-approximable, then the reduced free product of $(A_i,\phi_i)$ has the CCAP.
\end{theorem}
This theorem then provides a promising strategy for showing that free products of nuclear C*-algebra have the CCAP, and already covers many cases of interest.  Indeed, one  sees that all product states on UHF algebras are CP-approximable and that all states on a commutative C*-algebra, or a subalgebra of the compact operators are CP-approximable.

Therefore, we would like to know if all states on nuclear C*-algebras are CP-approximable.  In Section \ref{sec:nonCPapp}, we give an example of a faithful state on $M_2\otimes C[0,1]$ which is not CP-approximable.  On the other hand, in Section \ref{sec:GNSfaith} we show that if $A$ is a homogeneous C*-algebra and $\phi$ is a GNS-faithful state (Definition \ref{defn:GNSfaith}) then $\phi$ is CP-approximable.  This, combined with Theorem \ref{thm:ericandxu},  shows that the reduced free product of homogeneous C*-algebras with respect to GNS-faithful states have the CCAP. In Section \ref{sec:further} we discuss some open questions raised by this work.

We will use the following notation throughout: For a C*-algebra $A$, we let $A^+$ denote the positive cone and we write $M_n$ for the $n\times n$ matrices over $\mathbb{C}.$  For a linear map $T$ between C*-algebras we define the map $T^*$ as $T^*(x)=T(x^*)^*.$  For sets $X\subseteq Y$ we let $1_X$ denote the characteristic function of $X$ and $X^c$ denote the complement of $X.$ We write UCP as shorthand for the phrase ``unital completely positive.'' 
\section{A non CP-approximable faithful state on $M_2\otimes C[0,1]$} \label{sec:nonCPapp}
In this section, we  give an example of a faithful, non CP-approximable state on $M_2\otimes C[0,1].$  We first show that the approximating maps for CP-approximable diagonal states  can be manipulated to take a nice form. 

\subsection{Technical reformulation}
Let $n\in \mathbb{N}$ and $A$ be a unital C*-algebra. Let $e_{ij}$ denote the standard matrix units on $M_n\otimes 1_A\subseteq M_n\otimes A.$ 

\begin{defn} Let $\phi$ be a state on $M_n\otimes A.$  We say that $\phi$ is \textbf{diagonal} if $\phi(e_{ij}\otimes a)=0$ for all $a\in A$ whenever $i\neq j.$
\end{defn}

\begin{lemma} \label{lem:reformlem} Let A be a unital C*-algebra and $\phi$ a faithful, CP-approximable, diagonal state on $M_n\otimes A.$  Then, for every $1\leq i\leq j\leq n$ there are nets of finite rank maps $R_{ij,\alpha}:A\rightarrow A$ such that
\begin{equation*}
R_\alpha=\left[ \begin{array}{ccccc} R_{11,\alpha}& R_{12,\alpha}& \cdot &\cdot & R_{1n,\alpha}\\
R_{12,\alpha}^* & R_{22,\alpha}& \cdot &\cdot & R_{2n,\alpha}\\
\cdot&\cdot&\cdot&\cdot&\cdot\\
R_{1n,\alpha}^* &R_{2n,\alpha}^*&\cdot&\cdot& R_{nn,\alpha}\\ \end{array} \right],
\end{equation*}  
is UCP with   $\phi\circ R_\alpha=\phi$ for every $\alpha$ and $R_\alpha$ converges to the identity pointwise.
\end{lemma}
\begin{proof}
Let $S$ be a UCP map on $M_n\otimes A$ such that $\phi\circ S=\phi$ and  $e_{ii}S(e_{ii})e_{ii}\neq0$ for $i=1,...,n.$  Define the map $\widetilde{S}$ on $M_n\otimes A$ by
\begin{equation*}
\widetilde{S}(x)=\sum_{i,j=1}^n e_{ii}S(e_{ii}xe_{jj})e_{jj}
\end{equation*}
Then, clearly $\widetilde{S}$ is completely positive and maps $e_{ij}\otimes A$ into $e_{ij}\otimes A.$  We now slightly modify $\widetilde{S}$ to make it unital and $\phi$-preserving. For each $k=1,...,n$ define 
\begin{equation*}
T_k(x)=\frac{\phi(e_{kk}xe_{kk}-e_{kk}S(e_{kk}xe_{kk})e_{kk})}{\phi(e_{kk})}e_{kk}.
\end{equation*}
Since $\phi$ is diagonal and  $S$ is $\phi$-preserving, it follows that each $T_k$ is completely positive and that $R=\widetilde{S}+\sum_{k=1}^n T_k$ is $\phi$-preserving.  As in \cite[Proposition 4.5]{Ricard06}, we define
\begin{equation*}
\widetilde{T}_k(x)=\frac{1}{\V R(e_{kk})\V}R(e_{kk}xe_{kk})+\frac{\phi(e_{kk}xe_{kk})}{\phi(e_{kk})}\Big[e_{kk}-\frac{1}{\V R(e_{kk})\V}R(e_{kk})\Big]
\end{equation*}
Let $V$ be the diagonal matrix $diag(\V R(e_{11})\V^{-\frac{1}{2}},\V R(e_{22})\V^{-\frac{1}{2}},...,\V R(e_{nn})\V^{-\frac{1}{2}}).$  We finally define
\begin{equation*}
\widetilde{R}(x)=V R(x)V +\sum_{k=1}^n \widetilde{T}_k(x).
\end{equation*}
Then $\widetilde{R}$ is unital $\phi$-preserving and maps $e_{ij}\otimes A$ into $e_{ij}\otimes A$ for each $1\leq i,j\leq n.$
The conclusion of the lemma now follows from the above construction.
\end{proof}
\subsection{Construction of the example}
Let $m$ denote Lebesgue measure on $[0,1].$ By \cite{Rudin83}, there is a Lebesgue measurable subset $X\subseteq [0,1]$   such that for every non-empty open subset $\mathcal{O}\subseteq [0,1]$ we have 
\begin{equation}
m(X\cap\mathcal{O})>0,  \textrm{ and }m(X^c\cap\mathcal{O})>0. \label{eq:phifaithful}
\end{equation}
Define the diagonal state $\phi$ on $M_2\otimes C[0,1]$ as
\begin{equation}
\phi=\left[ \begin{array}{cc} 1_X \textrm{ }dm & 0\\ 0& 1_{X^c} \textrm{ }dm\\ \end{array} \right],
\end{equation}
i.e. $\phi((f_{ij})_{i,j})=\int_X f_{11}\textrm{ }dm+\int_{X^c} f_{22}\textrm{ }dm.$
By (\ref{eq:phifaithful}) it follows that $\phi$ is faithful on $M_2\otimes C[0,1].$
\begin{proposition} \label{prop:example} Let 
\begin{equation*}
S=\left[ \begin{array}{cc} S_{11}& S_{12}\\ S_{12}^* & S_{22}\\ \end{array} \right]
\end{equation*}
be a finite rank, $\phi$-preserving UCP map on $M_2\otimes C[0,1].$  Then $S_{12}=0.$ In particular, $\phi$ is not CP-approximable.
\end{proposition}
\begin{proof} Let $\theta\in [0,1]$ and $\epsilon>0.$ Choose an  open cover $I_1,...,I_N$ of $[0,1]$ such that for every $f\in \textrm{span}\{S_{11}(C[0,1]),S_{22}(C[0,1]) \}$ we have 
\begin{equation}
\max_{i=1,...,N}\sup_{x,y\in I_k}|f(x)-f(y)|<\V f\V\epsilon.\label{eq:bunchingf}
\end{equation}
Without loss of generality, suppose that $\theta\in I_1$ and there is an open interval $J$ such that
\begin{equation}
\theta\in J\subseteq I_1 \quad \textrm{and}\quad J\cap I_j=\emptyset \textrm{ for all }j\geq 2. \label{eq:Jdef}
\end{equation}
Moreover, assume that there are points $t_i\in I_i\setminus (\cup_{j\neq i}I_j)$ for $2\leq i\leq N.$
Let $\rho_1,\rho_2,...,\rho_N$ be a partition of unity subject to the cover $\{I_i\}_{i=1}^N.$
For $i=1,2$ define UCP maps $T_i:C[0,1]\rightarrow C[0,1]$ as 
\begin{equation*}
T_1(f)=\sum_{i=2}^N f(t_i)\rho_i + \Big( \frac{\int_X f\rho_1\textrm{ } dm}{\int_X \rho_1\textrm{ } dm}\Big)\rho_1\textrm{ } ,\quad T_2(f)=\sum_{i=2}^N f(t_i)\rho_i + \Big( \frac{\int_{X^c} f\rho_1\textrm{ } dm}{\int_{X^c} \rho_1\textrm{ } dm}\Big)\rho_1.
\end{equation*}
By (\ref{eq:bunchingf}), we have
\begin{equation}
\V T_i(S_{ii}(f))-S_{ii}(f)   \V\leq \epsilon \V f\V \quad \textrm{ for all  } f\in C[0,1]\textrm{ and }i=1,2. \label{eq:1}
\end{equation}
Let $0\leq f\leq 1$ be in $C[0,1].$  Then $\left[ \begin{array}{ll}f & f \\
f & 1 \\
\end{array} \right]\geq0.$ Since $S_{22}$ is unital,  by (\ref{eq:1}), we have
\begin{equation*}
T_1(S_{11}(f))-|S_{12}(f)|^2 \geq det\left[ \begin{array}{ll} S_{11}(f) & S_{12}(f) \\
S_{12}(f)^* & 1 \\
\end{array} \right] -\epsilon\geq -\epsilon
\end{equation*}
By a similar calculation with $T_2$, we have
\begin{equation}
|S_{12}(f)|^2\leq T_i(S_{ii}(f))+\epsilon \quad \textrm{for }i=1,2. \label{eq:2}
\end{equation}
Now, let $\theta\in I\subseteq J$ (as defined in (\ref{eq:Jdef})) be an  interval and set $\phi_I=\frac{1}{m(I)}\int _I f\textrm{ }dm.$  Since $\rho_i=0$ on $I$ for every $i>1$ and $\int_X S_{11}(f)dm=\int_X f dm$, we have
\begin{align}
\phi_I(T_1(S_{11}(f))) & =\phi_I\Big( \frac{\int_XS_{11}(f)\rho_1\textrm{ }dm}{\int_X\rho_1\textrm{ }dm}  \rho_1 \Big) \notag \\
&\leq \phi_I\Big( \frac{\int_XS_{11}(f)\textrm{ }dm}{\int_X\rho_1\textrm{ }dm}  \rho_1 \Big) \notag \\
&= \phi_I\Big( \frac{\int_Xf\textrm{ }dm}{\int_X\rho_1\textrm{ }dm}  \rho_1 \Big). \label{al:3} 
\end{align}
Let $K\subseteq [0,1]$ be compact and $\mathcal{O}\subseteq [0,1]$ be open  such that $K\subseteq X\subseteq \mathcal{O}$ and 
\begin{equation}
m(\mathcal{O}\setminus K)<\epsilon\min\Big\{ \int_X \rho_1 dm,\int_{X^c} \rho_1 dm  \Big\}. \label{eq:epbound}
\end{equation}
Let $0\leq\gamma\leq1$ be a continuous function such that $\gamma|_K=1$ and $\gamma|_{\mathcal{O}^c}=0.$
By (\ref{al:3}) and (\ref{eq:epbound}), it follows that
\begin{equation*}
\phi_I(T_1(S_{11}((1-\gamma) f)))\leq \epsilon.
\end{equation*}
By a similar computation as in (\ref{al:3}), we obtain
\begin{equation*}
\phi_I(T_2(S_{22}(\gamma f)))\leq \epsilon.
\end{equation*}
We now combine the previous inequalities with (\ref{eq:2}) to obtain
\begin{align*}
\frac{1}{m(I)}\int_I |S_{12}(f)|^2\textrm{ }dm&\leq \Big( \Big(\frac{1}{m(I)}\int_I |S_{12}(\gamma f)|^2\textrm{ }dm\Big)^{1/2}+\Big(\frac{1}{m(I)}\int_I |S_{12}((1-\gamma) f)|^2\textrm{ }dm\Big)^{1/2}\Big)^2\\
&\leq (\sqrt{2\epsilon}+\sqrt{2\epsilon})^2.
\end{align*}
Since $\epsilon>0$ was arbitrary, it follows that
\begin{equation*}
|S_{12}(f)|^2(\theta)=\lim_{\theta\in I;\textrm{ }m(I)\rightarrow 0}\frac{1}{m(I)}\int_I |S_{12}(f)|^2 \textrm{ }dm=0.
\end{equation*}
Since $\theta$ was arbitrary it follows that $S_{12}=0.$ That $\phi$ is not CP-approximable now follows from Lemma \ref{lem:reformlem}.
\end{proof}
\section{GNS-faithful states} \label{sec:GNSfaith}
In this section we show that all states of the following type on homogeneous C*-algebras are CP-approximable:
\begin{defn} \label{defn:GNSfaith}
Let B be a unital C*-algebra and $\phi$ a state on B with associated GNS representation and cyclic vector $(\pi_\phi,\xi_\phi).$  We say that $\phi$ is \textbf{GNS-faithful} if the vector state $\la \textrm{ }(\cdot)\textrm{ }\xi_\phi, \xi_\phi\ra$ is faithful on $\pi_\phi(B)''.$ 
\end{defn}
For a function $g:X\rightarrow Y$ between sets,  we will adopt the common practice from measure theory by writing $\{ \phi(g)\}$ as shorthand notation for $\{x\in X : \phi(g(x)) \}$ where $\phi$ is some formula (for example $\phi(x)=\textrm{``}x\leq 10\textrm{"}$ or ``$x$ is invertible."). 

The following is easy to verify from the Riesz representation theorem, Radon-Nikodym theorem and basic facts about the GNS construction.
\begin{lemma} \label{lem:statechange} Let $X$ be a compact Hausdorff space, $n\in\mathbb{N}$ and $\phi$ a state on $A=M_n\otimes C(X).$ Let $\tau$ denote the tracial state on $M_n.$  There is a regular Borel probability measure $\mu$ on $X$ and a bounded measurable function $g:X\rightarrow M_n^+$ such that 
\begin{equation}
\phi(f)=\int_X \tau(g(x)f(x))d\mu \quad \textrm{for }f\in A.\label{eq:defofphi}
\end{equation}
Moreover, if $\phi$ is GNS-faithful, then the natural extension of $\phi$ to $M_n\otimes L^\infty(X,\mu)$ is faithful. In particular,
\begin{equation}
\mu\Big(\{g \textrm{ is invertible }\}\Big)=1. \label{eq:biginvert}
\end{equation} 
\end{lemma}
\begin{theorem} \label{thm:GNSfaithful} Let $X$ be a compact Hausdorff space, $n\in\mathbb{N}$ and $\phi$ be a GNS-faithful state on $A=M_n\otimes C(X).$ Then $\phi$ is CP-approximable.
\end{theorem}
\begin{proof} Let $\mathcal{F}\subseteq C(X)$ be a finite subset and $\epsilon>0.$  We construct a finite rank, UCP, $\phi$-preserving map $T:A\rightarrow A$ such that 
\begin{equation}
\max_{f\in \mathcal{F}}\max_{b\in M_n, \V b\V\leq1}\V T(b\otimes f)-b\otimes f \V\leq \epsilon. \label{eq:whattoprove}
\end{equation}
To this end, let $\mathcal{B}$ be a finite set of disjoint Borel subsets of $X$ with positive measure and $\cup \mathcal{B}=X$  such that 
\begin{equation}
\max_{f\in\mathcal{F}}\max_{B\in\mathcal{B}}\sup_{x,y\in B}|f(x)-f(y)|\leq \epsilon/8.\label{eq:fconstant}
\end{equation}
Let $\mu, \tau$ and $g$ be as in Lemma \ref{lem:statechange}. Set $\V g\V:=\sup_{x\in X}\V g(x) \V_{M_n}.$  For each $\delta>0$ define
\begin{equation*}
Z_\delta=\{a\in M_n^+: \delta\leq  a \leq \V g \V   \}.
\end{equation*}
For a subset $W\subseteq M_n$, we let $\overline{\textrm{conv}}(W)$ denote the closed convex hull of $W.$  Partition $Z_\delta$ into finitely many disjoint Borel subsets $Z_{1,\delta},...,Z_{N_\delta, \delta}$ such that
\begin{equation}
\max_{i=1,..,N_\delta}\sup\{\V a^{\frac{1}{2}}b^{-\frac{1}{2}}-1 \V: a,b\in \overline{\textrm{conv}}(Z_{i,\delta}) \}<\frac{\epsilon}{8}. \label{eq:boundsonsquares}
\end{equation}
By (\ref{eq:biginvert}) we may  suppose that $\delta$ has been chosen small enough so that for every $B\in \mathcal{B}$ we have $\mu(B\cap\{ g\geq \delta\})>0.$  Then for each $B\in\mathcal{B}$ there is an index $1\leq i_B\leq N_\delta$ such that
\begin{equation}
0<r:= \min_{B\in\mathcal{B}} \mu(B\cap \{g \in Z_{i_B, \delta}\}). \label{eq:goodmeasuresets}
\end{equation}
Let $F:(0,1)\rightarrow (0,1)$ be a function that satisfies the following sentence:
\begin{equation} 
(\forall a,b\in M_n^+)(\textrm{sp}(a),\textrm{sp}(b) \subseteq [s,\V g\V])(\V a-b\V\leq F(s)\rightarrow \V a^{-\frac{1}{2}}-b^{-\frac{1}{2}} \V\leq \epsilon/(\V g\V8)) \label{eq:invertroots}
\end{equation}
By (\ref{eq:biginvert}) there is  a $\gamma<\delta$ such that
\begin{equation}
\mu(\{ g\geq \gamma \}^c)<\frac{1}{2}\min\{F((\delta r)/2),\frac{\epsilon\delta r}{8\V g\V} \}.\label{eq:boundingbadset}
\end{equation}
 Now obtain a finite partition of $Z_\gamma$; $Z_{1,\gamma}, ...,Z_{N_\gamma,\gamma}$ that satisfies (\ref{eq:boundsonsquares}) and such that $Z_{i,\delta}=Z_{i,\gamma}$ for all $1\leq i\leq N_\delta.$
For each $B\in\mathcal{B}$ define the set
\begin{equation}
Y_{1,B}:= B\cap (\{g\geq \gamma\}^c\cup \{ g\in Z_{i_B,\delta} \}) \label{eq:defnofY1B}
\end{equation}
Then let $Y_{2,B},...,Y_{N_B,B}$ be an enumeration of the sets
\begin{equation}
\{ B\cap \{g\in Z_{i,\gamma}\} : \mu(B\cap \{g\in Z_{i,\gamma}\})>0 \textrm{ and }i\neq i_B\}.
\end{equation}
For each $B\in\mathcal{B}$ and $1\leq i\leq N_B$ obtain a compact  set $K_{i,B}\subseteq Y_{i,B}$ and an open set $O_{i,B}\supseteq Y_{i,B}$ such that 
\begin{equation}
\mu(O_{i,B}\setminus K_{i,B})<\frac{1}{2} \min\Big\{F((\gamma \mu(Y_{i,B}))/2),\frac{\epsilon\gamma \mu(Y_{i,B})}{8\V g\V} \Big\}. \label{eq:approxsetwithcptopen}
\end{equation}
Then, 
\begin{equation}
K_{i,B}\cap K_{i',B'}=\emptyset \quad\textrm{ if }\quad(i,B)\neq (i',B') \quad \textrm{and }\quad\bigcup_{B\in \mathcal{B}}\bigcup_{i=1}^{N_B} O_{i,B}=X. \label{eq:disjointpluscover}
\end{equation}
Let $\{ \rho_{i,B}: B\in\mathcal{B}, 1\leq i\leq N_B\}\subseteq C(X)\subseteq M_n\otimes C(X)$ be a partition of unity subject to the open cover $\{ O_{i,B}\}$.  Note that 
\begin{equation}
\rho_{i,B}(x)=1\quad \textrm{ for every }x\in K_{i,B}.
\end{equation}
Denote by  $\mathbb{E}:M_n\otimes L^\infty(X,\mu)\rightarrow M_n$ the conditional expectation $id_{M_n}\otimes\int (\cdot)d\mu.$  Note that $\mathbb{E}$ satisfies the following equations:
\begin{equation}
\tau(a\mathbb{E}(h))=\int \tau(ah(x))d\mu(x) \quad \textrm{for all }a\in M_n\quad \textrm{and }h\in M_n\otimes L^\infty(X,\mu). \label{eq:defofE}
\end{equation}
Before constructing our map $T$, we first need a few estimates.  Let $B\in \mathcal{B}.$ Due to the difference in definitions of $Y_{1,B}$ and $Y_{i,B}$ for $i>1$, we first isolate the case $i=1.$   We have
\begin{align*}
\E (\rho_{1,B} g)&=\E \Big(1_{K_{1,B}\cap \{ g\in Z_{i_B,\delta}\}} \rho_{1,B}g \Big)+\E\Big(1_{({K_{1,B}\cap \{ g\in Z_{i_B,\delta}\}})^c}\rho_{1,B}g\Big)\\
&\geq \delta(r-\mu(O_{1,B}\setminus K_{1,B})-\mu(\{ g\geq\gamma\}^c))+\E\Big(1_{({K_{1,B}\cap \{ g\in Z_{i_B,\delta}\}})^c}\rho_{1,B}g\Big)\\
&\geq \frac{\delta r}{2}+\E\Big(1_{({K_{1,B}\cap \{ g\in Z_{i_B,\delta}\}})^c}\rho_{1,B}g\Big)\\
\end{align*} 
From this we deduce that
\begin{equation}
\textrm{sp}(\E(\rho_{1,B}g))\subseteq [\frac{\delta r}{2}, \V g\V]\quad \textrm{ and }\quad \V \E(\rho_{1,B}g)^{-\frac{1}{2}}\V\leq \Big(\frac{2}{\delta r}\Big)^{\frac{1}{2}}. \label{eq:fixingspect}
\end{equation} 
Now, set $a_{1,B}:=\E \Big(1_{K_{1,B}\cap \{ g\in Z_{i_B,\delta}\}} \rho_{1,B}g \Big).$  
By (\ref{eq:boundingbadset}) and (\ref{eq:approxsetwithcptopen}) it follows that 
\begin{equation*}
\V \E(\rho_{1,B}g)-a_{1,B}\V=\V \E\Big(1_{({K_{1,B}\cap \{ g\in Z_{i_B,\delta}\}})^c}\rho_{1,B}g\Big) \V\leq F((\delta r)/2).
\end{equation*}
From which, in conjunction with (\ref{eq:invertroots}) and (\ref{eq:fixingspect}), we deduce that
\begin{equation}
\V (\E(\rho_{1,B}g))^{-\frac{1}{2}}-a_{1,B}^{-\frac{1}{2}}\V<\epsilon/(8\V g \V). \label{eq:aiBclose}
\end{equation} 
Notice that  $\mu(K_{1,B}\cap \{ g\in Y_{i_B,\delta}\})^{-1}a_{1,B}\in \overline{\textrm{conv}}(Z_{i_B,\delta}).$ So by (\ref{eq:boundsonsquares}) and (\ref{eq:aiBclose}) we obtain
\begin{equation}
\sup_{y\in Z_{i_B,\delta}}\Big\V g(y)^{\frac{1}{2}}\E(\rho_{1,B}g)^{-\frac{1}{2}}-   \mu(K_{1,B}\cap \{ g\in Z_{i_B,\delta}\})^{-\frac{1}{2}} \Big\V <\frac{\epsilon}{4}\mu(K_{1,B}\cap \{ g\in Z_{i_B,\delta}\})^{-\frac{1}{2}}. \label{eq:gandEinv}
\end{equation}
We isolated the case $i=1$ above, because we had to take into account the sets $B\cap \{ g\geq \gamma\}^c.$  When $i>1$, we no longer concern ourselves with these sets. Therefore, by essentially the same (but slightly simpler) estimates as above and using line (\ref{eq:approxsetwithcptopen}) we obtain
\begin{equation}
\max_{B\in\mathcal{B}}\max_{i=2,..,N_B}\sup_{y\in g^{-1}(Y_{i,B})}\Big\V g(y)^{\frac{1}{2}}\E(\rho_{i,B})^{-\frac{1}{2}}-   \mu(K_{i,B})^{-\frac{1}{2}} \Big\V <\frac{\epsilon}{4}\mu(K_{i,B})^{-\frac{1}{2}},
\end{equation}
and
\begin{equation}
\V \E(\rho_{i,B}g)^{-\frac{1}{2}}\V\leq \Big(\frac{2}{\gamma \mu(Y_{i,B})}\Big)^{\frac{1}{2}}.
\end{equation}
We now define the map $T:A\rightarrow A$ as
\begin{equation}
T(h)=\sum_{B\in\mathcal{B}}\sum_{i=1}^{N_B} \Big[\E(\rho_{i,B} g)^{-\frac{1}{2}}\E(\rho_{i,B} g^{\frac{1}{2}}hg^{\frac{1}{2}})\E(\rho_{i,B} g)^{-\frac{1}{2}}\Big]\otimes \rho_{i,B}.\label{eq:defnofT}
\end{equation}
  It is clear that $T$ is UCP and finite rank.  We now show that  $T$ preserves $\phi.$  To this end, fix a pair $(i,B)$.  Then by (\ref{eq:defofphi}) and (\ref{eq:defofE}) we have
\begin{align*}
&\phi\Big(\Big[\E(\rho_{i,B} g)^{-\frac{1}{2}}\E(\rho_{i,B} g^{\frac{1}{2}}hg^{\frac{1}{2}})\E(\rho_{i,B} g)^{-\frac{1}{2}}\Big]\otimes \rho_{i,B}\Big)\\
=& \int \tau\Big( \E(\rho_{i,B} g)^{-\frac{1}{2}}\E(\rho_{i,B} g^{\frac{1}{2}}hg^{\frac{1}{2}})\E(\rho_{i,B} g)^{-\frac{1}{2}}\rho_{i,B}(x)g(x)\Big)d\mu(x)\\
=& \tau\Big(\E(\rho_{i,B} g)^{-\frac{1}{2}}\E(\rho_{i,B} g^{\frac{1}{2}}hg^{\frac{1}{2}})\E(\rho_{i,B} g)^{-\frac{1}{2}}\E(\rho_{i,B}g)         \Big)\\
=&\tau\Big(   \E(\rho_{i,B} g^{\frac{1}{2}}hg^{\frac{1}{2}})   \Big)\\
=& \int\tau\Big( \rho_{i,B}(x)g(x)h(x)\Big)d\mu(x).
\end{align*}
Since the $\rho_{i,B}$ form a partition of unity, it follows that $\phi\circ T=\phi.$

We now show that $T$ approximately fixes matrices. Let $b\in M_n$ of norm 1 and $a\in M_n$ with $\tau(|a|)\leq1$. Again, we will isolate the case $i=1$ for each $B\in\mathcal{B}.$ For convenience we denote $\mathcal{S}=K_{1,B}\cap \{ g\in Y_{i_B,\delta}\}$ and  $s:=\mu(\mathcal{S}).$ Then,
\begin{align*}
&\Big|\tau\Big(a\E(\rho_{1,B} g)^{-\frac{1}{2}}\E(\rho_{1,B} g^{\frac{1}{2}}bg^{\frac{1}{2}})\E(\rho_{1,B} g)^{-\frac{1}{2}}\Big) -\tau(ab) \Big|\\
&=\Big|\int_{O_{1,B}}\tau\Big(a\E(\rho_{1,B} g)^{-\frac{1}{2}}\rho_{1,B}(y) g(y)^{\frac{1}{2}}bg(y)^{\frac{1}{2}}\E(\rho_{1,B} g)^{-\frac{1}{2}}  \Big)d\mu(y)-\tau(ab)\Big|\\
&\leq \Big| \int_{\mathcal{S}}\tau\Big(a  \E(\rho_{1,B} g)^{-\frac{1}{2}}g(y)^{\frac{1}{2}}bg(y)^{\frac{1}{2}}\E(\rho_{1,B} g)^{-\frac{1}{2}}  \Big)    -\tau(ab)\Big|+\mu(O_{1,B}\setminus\mathcal{S})\V g\V\V\E(\rho g)^{-\frac{1}{2}}  \V^2\\
&\leq \Big| \int_{\mathcal{S}}\tau\Big(a  \E(\rho_{1,B} g)^{-\frac{1}{2}}g(y)^{\frac{1}{2}}bg(y)^{\frac{1}{2}}\E(\rho_{1,B} g)^{-\frac{1}{2}}  \Big)    -\tau(ab)\Big|+\frac{\epsilon}{4} \quad (\textrm{by }(\ref{eq:boundingbadset} ,\ref{eq:approxsetwithcptopen}, \ref{eq:fixingspect}))\\
&\leq \Big| s^{-1}\int_{\mathcal{S}}\tau(ab)-\tau(ab)\Big| +2s\sup_{y\in \mathcal{S}}\V g(y)^{\frac{1}{2}}\E(\rho_{1,B}g)^{-\frac{1}{2}}\V\V g(y)^{\frac{1}{2}}\E(\rho_{1,B}g)^{-\frac{1}{2}}-s^{-\frac{1}{2}} \V+\frac{\epsilon}{4}\\
&\leq \Big|s^{-1}\int_{\mathcal{S}}\tau(ab)d\mu-\tau(ab)\Big|+2s(s^{-\frac{1}{2}}(1+\frac{\epsilon}{4})s^{-\frac{1}{2}}\frac{\epsilon}{4})+\frac{\epsilon}{4} \quad (\textrm{by } (\ref{eq:gandEinv}))\\
&\leq \frac{7\epsilon}{8}.
\end{align*}
Once again, by very similar (but slightly simpler) estimates we obtain
\begin{equation}
\max_{B\in\mathcal{B}}\max_{i=2,...,N_B}\Big|\tau\Big( a\E(\rho_{i,B} g)^{-\frac{1}{2}}\E(\rho_{i,B} g^{\frac{1}{2}}bg^{\frac{1}{2}})\E(\rho_{i,B} g)^{-\frac{1}{2}}\Big) -\tau(ab)\Big|\leq \frac{7\epsilon}{8}.
\end{equation}
Since $a$ and $b$ were arbitrary, we obtain the following:
\begin{equation*}
\max_{B\in\mathcal{B}}\max_{i=1,...,N_B}\sup_{b\in M_n, \V b\V=1}\Big\V  \E(\rho_{i,B} g)^{-\frac{1}{2}}\E(\rho_{i,B} g^{\frac{1}{2}}bg^{\frac{1}{2}})\E(\rho_{i,B} g)^{-\frac{1}{2}} -b\Big\V_{M_n}\leq\frac{7\epsilon}{8}.
\end{equation*}
From this it follows that $\V T(b)-b\V\leq \frac{7\epsilon}{8}$ for every $b\in M_n$ of norm 1.

Finally, since all of the functions $f\in\mathcal{F}$ have variation at most $\epsilon/8$ on the sets $B\in\mathcal{B}$, very easy estimates show that
$\V T(b\otimes f)-b\otimes f \V\leq \epsilon$ for every $f\in\mathcal{F}$ and $b\in M_n$ of norm 1. This completes the proof.
\end{proof}
\begin{cor} Let $(A_i,\phi_i)_{i\in I}$ be a family of homogeneous C*-algebra such that each $\phi_i$ is GNS faithful.  Then the reduced free product of the family $(A_i,\phi_i)_{i\in I}$ has the CCAP.
\end{cor}
\begin{proof} This follows from Theorem \ref{thm:GNSfaithful} and Theorem \ref{thm:ericandxu}
\end{proof}

\section{Further Remarks} \label{sec:further}
Proposition \ref{prop:example} raises the following natural question:
\begin{question} \label{ques:charCPap} Let A be a nuclear C*-algebra.  Is there a nice characterization of CP-approximable states?
\end{question}
In general, we feel that there will be no satisfying answer to Question \ref{ques:charCPap} because the set of CP-approximable states needn't be  convex or norm closed. Indeed consider $\phi$ from Proposition \ref{prop:example}.  It is easy to see that the states
\begin{equation*}
\phi_1=2\left[ \begin{array}{cc} 1_X \textrm{ }dm & 0\\ 0& 0 \\ \end{array} \right],\quad \phi_2=2\left[ \begin{array}{cc} 0  & 0\\ 0& 1_{X^c} \textrm{ }dm\\ \end{array} \right]
\end{equation*}
are both CP-approximable, but $\phi=1/2(\phi_1+\phi_2)$ is not CP-approximable.  Furthermore, Theorem \ref{thm:GNSfaithful} shows that $\phi$ is the norm limit of CP-approximable states. 

It is fairly straightforward to see that all states on a commutative C*-algebra or on a unital subalgebra of $K(H)^1$ (the unitization of the compact operators) are CP-approximable. Since $M_2\otimes C[0,1]$ has a non CP-approximable state, we feel the answer to the following should be yes:
\begin{question}
Let A be a separable nuclear C*-algebra such that every state is CP-approximable.  Is A commutative or isomorphic to a unital subalgebra of $K(H)^1$?
\end{question}
We note that the separability condition is necessary as it is straightforward to show (using ideas similar to those in Theorem \ref{thm:GNSfaithful}) that every state on $M_n\otimes L^\infty[0,1]$ is CP-approximable. It is the abundance of projections in $M_n\otimes L^\infty[0,1]$ that allows one to show that all states are CP-approximable, hence it is natural to consider if all states on a real rank zero C*-algebra are CP-approximable.  Applying a deep theorem of  Kirchberg we see that real rank zero doesn't help:
\begin{cor} Let $M_{2^\infty}$ be the CAR algebra.  Then there is a state $\psi$ on $M_{2^\infty}$ that is not CP-approximable.  Since $M_{2^\infty}$ is simple, the GNS representation associated with $\psi$ is necessarily faithful.
\end{cor}
\begin{proof}
By \cite[Corollary 1.5]{Kirchberg95} there are unital completely positive maps $T:M_2\otimes C[0,1]\rightarrow M_{2^\infty}$ and $S:M_{2^\infty}\rightarrow M_2\otimes C[0,1]$ such that $ST=id.$  By Proposition \ref{prop:example}, it follows that $\psi=\phi\circ S$ is not CP-approximable.
\end{proof}
Finally, Proposition \ref{prop:example} leaves open the natural question:
\begin{question}
Does the reduced  free product of $(M_2\otimes C[0,1],\phi)$ with itself have the CCAP?
\end{question}
\section*{Acknowledgement} J'aimerais remercier \'Eric Ricard, Jean Roydor et Quanhua Xu pour les discussions enrichissantes concernant ces travaux.

\end{document}